\documentclass[12pt,leqno,fleqn]{amsart}  
\usepackage{amsmath,amstext,amsthm,amssymb,amsxtra}
\usepackage{txfonts} 
\usepackage[T1]{fontenc}
\usepackage{lmodern}

 \usepackage{euler}   

\usepackage{mathtools}
\mathtoolsset{showonlyrefs,showmanualtags}

\usepackage{hyperref} 
\hypersetup{
    colorlinks=true,       
    linkcolor=blue,          
    citecolor=magenta,        
    filecolor=magenta,      
    urlcolor=cyan           
}

\usepackage{amsrefs}  

\allowdisplaybreaks

\theoremstyle{plain} 
\newtheorem{lemma}[equation]{Lemma} 
\newtheorem{proposition}[equation]{Proposition} 
\newtheorem{theorem}[equation]{Theorem} 
 
\newtheorem{conjecture}[equation]{Conjecture}
\newtheorem{priorResults}{Theorem}

\theoremstyle{definition}
\newtheorem{definition}[equation]{Definition} 

\theoremstyle{remark}
\newtheorem{remark}[equation]{Remark}

\numberwithin{equation}{section}

\title[Separated Bumps and CZOs] {On the Separated Bumps Conjecture for \\ Calder\'on-Zygmund Operators}
\author{Michael T. Lacey}   

\address{ School of Mathematics, Georgia Institute of Technology, Atlanta GA 30332, USA}
\email {lacey@math.gatech.edu}
\thanks{Research supported in part by grant NSF-DMS 1265570 
and the Australian Research Council through grant ARC-DP120100399.}

\begin{document}

\begin{abstract}
Let $ \sigma (dx) = \sigma (x)dx$ and $ w (dx)= w (x)dx$ be two weights with non-negative locally finite densities 
on $ \mathbb R ^{d}$, and let $ 1<p< \infty $. 
A sufficient condition for the norm estimate 
\begin{equation*}
\int \lvert  T (\sigma f)\rvert ^{p} \, w (dx) \le C_{T, \sigma ,w} ^{p}  \int \lvert  f\rvert ^{p}\, \sigma (dx)  , 
\end{equation*}
valid for all Calder\'on-Zygmund operators $ T$ is that the condition below holds.  
\begin{equation*}
\sup _{\textup{$ Q$ a cube}} 
\lVert  \sigma ^{\frac1{p'}}\rVert_{L ^{A} (Q, \frac {dx}{\lvert  Q\rvert})}  \varepsilon  (\lVert  \sigma ^{\frac1{p'}}\rVert_{L ^{A} (Q, \frac {dx}{\lvert  Q\rvert})}/ 
\sigma (Q) ^{\frac1{p'}})   \bigl[\frac {w (Q) } {\lvert  Q\rvert } \bigr] ^{\frac1{p}} < \infty  
\end{equation*}
Here $ A$ is Young function, with dual  in the P{\'e}rez class $ B_{p}$, 
and the function $ \varepsilon (t)$ is  increasing  on $ (1, \infty )$ with $ \int  ^{\infty } \varepsilon (t) ^{-p'} \frac {dt} t < \infty $. 
Moreover, a  dual condition holds, with the roles of the weights and $ L ^{p}$ indices reversed also holds.   
This is an alternate version of a  result of Nazarov, Reznikov and Volberg  ($ p=2$), one with a simpler formulation, and proof based upon 
stopping times. 
\end{abstract}
 
\maketitle

\section{Introduction} 
Our subject is the two weight theory for Calder\'on-Zygmund operators.  To say that $ \sigma $ is a \emph{weight} on $ \mathbb R ^{d}$ 
means that it is a locally finite, non-negative, Borel measure. Throughout, we further assume that $ \sigma $ is absolutely continuous with respect to Lebesgue measure, $ \sigma (dx)=\sigma (x)dx$, and we will not distinguish the density from the weight.  Fix $ 1<p< \infty $. 
For which pairs of weights $ w$ and $ \sigma $ does the following hold for all Calder\'on-Zygmund operators $ T$? 
\begin{equation} \label{e:T}
\int \lvert  T (\sigma f)\rvert ^{p} \, w (dx) \le C_{T, \sigma ,w} ^{p}  \int \lvert  f\rvert ^{p}\, \sigma (dx)  
\end{equation}
The precise definition for these and other terms will be given in the next section.  Below, we will write $ T _{\sigma } f := T (\sigma f)$. 
In the case that  $ w (x) >0 $ a.e., and $ \sigma (x)=w (x) ^{1-p'}$, so that $ w (x) ^{\frac1{p}} \sigma (x) ^{\frac1{p'}} \equiv 1$, where $ 1= \frac1{p} + \frac1{p'}$, the necessary and sufficient condition is  the famous Muckenhoupt condition, 
\begin{equation} \label{e:Ap}
[w] _{A_p} := \sup _{Q} \frac {w (Q)} {\lvert  Q\rvert } \bigl[ \frac {\sigma (Q)} {\lvert  Q\rvert } \bigr] ^{p-1} < \infty . 
\end{equation}
In the classical case of just the Hilbert transform, this is the result of Hunt-Muckenhoupt-Wheeden \cite{MR0312139}. 
Thoroughly modern treatments of the $ A_p$ theory are given in \cites{12123840,2912709,MR3085756}.  
But, the two weight problem, in which $ \sigma $ and $ w$ have no explicit relationship, is far more complicated, as was recognized even at early stages of the subject.

The question addressed here is: What is the sharpest condition, expressible in terms similar to the $ A_p$ condition, which is sufficient for the 
family of inequalities \eqref{e:T}? 
There are many results in the literature which modify \eqref{e:Ap} by requiring slightly higher integrability conditions 
on the densities of the weights.  The sharpest of these conditions is the  `separated bump conjecture,' which arises from work of 
Cruz-Uribe and  P{\'e}rez \cites{MR1713140,MR1793688}, and Cruz-Uribe and  Reznikov and  Volberg \cite{11120676}.

To set the stage for this condition, observe that the $ A_p$ condition can be rewritten as the $ p$th power of 
\begin{equation*}
\sup _{Q}  \lVert w ^{\frac1{p}}\rVert_{L ^{p} (Q, dx/\lvert  Q\rvert )} \cdot \lVert \sigma ^{\frac1{p'}}\rVert_{L ^{p'} (Q, dx/\lvert  Q\rvert )}.  
\end{equation*}
The bump conditions, in the two weight setting, strengthen this condition by replacing the $ L ^{p}$ norms by stronger ones.
The strengthening  is on the scale of Orlicz spaces.  
Set $ \langle f \rangle _{Q}=\lvert  Q\rvert ^{-1} \int _{Q} f \, dx $, which is just the average of $ f$.   Write 
\begin{equation*}
\langle  f \rangle _{A, Q} := \lVert  f\rVert_{ L ^{A} (Q, \frac {dx} {\lvert  Q\rvert })}, 
\end{equation*}
where $ A$ is  a Young function $ A \::\: [0, \infty ) \to [0, \infty )$, 
and  $ L ^{A}$ denotes the  Luxembourg norm associated to $ A$,  and we have normalized Lebesgue measure on $ Q$.  
The \emph{dual} Young function  to $ A$ is denoted by $ \overline A$.  

We write $ A  \in B_p$ if 
\begin{equation}\label{e:Bp}
\int ^{\infty } \frac {A (t)} {t ^{p}} \cdot \frac {dt} t < \infty . 
\end{equation}
A typical example of function $ A (t)$ is  just a bit smaller than $ t ^{p}$, for example $ t ^{p} (\log t) ^{-1 - \epsilon }$, for $ \epsilon >0$, 
in which case $ \langle f \rangle _{A,Q}$ is just a bit smaller than $ \langle \lvert  f\rvert ^{p}  \rangle_Q ^{1/p}$. 
To put the main conjecture in context, we recall the result of Carlos P{\'e}rez which is basic for us. 

\begin{priorResults}[C. P{\'e}rez \cite{MR1260114}]   Suppose that $  A \in B_p$.  Then, there holds 
\begin{equation} 
\label{e:cmf}
 \bigl\lVert  \sup _{ \textup{$ Q$ a cube}}  \mathbf 1_{Q}  \langle  f \rangle _{A, Q}\rVert_{p} \lesssim \lVert f\rVert_{p}.  
\end{equation}
Moreover, the condition $ A\in B_p$ is necessary for the inequality above. 
\end{priorResults}

This has the following implication for weighted inequalities.  The Theorem below from \cite{MR1260114} gives the sharpest `$ A_p$-like' sufficient conditions for the boundedness of the maximal function.  

\begin{priorResults}[C. P{\'e}rez, \cite{MR1260114}]  \label{t:cmf}  Suppose that $ 1< p < \infty $, and $  A \in B_{p}$.  Assume the condition below 
involving the dual function $ \overline  A$,
\begin{equation}\label{e:[]}
[\sigma , w] _{\overline A,p'} := \sup _{ \textup{$ Q$ a cube}}  \langle \sigma ^{\frac1{p'}}  \rangle_{\overline A, Q} \langle w \rangle_ Q ^{\frac1{p}}< \infty .
\end{equation}
Then this maximal function inequality  holds: $ \lVert M _{\sigma }f\rVert_{L ^{p} (w)} \lesssim \lVert f\rVert_{L ^{p} (\sigma )}$, 
where 
\begin{equation*}
M _{\sigma } f = \sup _{Q} \mathbf 1_{Q} \langle  \lvert  f\rvert  \cdot \sigma  \rangle_Q .  
\end{equation*}
\end{priorResults}

Since $  A\in B_p$, $  A$ must be a bit smaller than $ t ^{p}$, hence $\overline A$ must be a bit bigger than $ t ^{p'}$. That is we will have 
$ \lVert \sigma ^{1/{p'}} \rVert_{\overline A, Q} \gtrsim \langle \sigma  \rangle_Q^{1/{p'}}$, whence the condition  \eqref{e:[]} is stronger than the two weight  $ A_p$ condition.  
(It is known that the two weight $ A_p$ condition only characterizes the weak-type norm of the maximal function.) 

This Theorem answers a question of Cruz-Uribe and  P{\'e}rez \cite{MR1793688}. It was proved almost at the same time by 
Nazarov, Reznikov and  Volberg \cite{MR3127385} ($ p=2$), and by Lerner \cite{MR3085756} for all $ p$.  

\begin{priorResults} \label{t:bumps} Let $ \sigma $ and $ w$ be two weights with densities, and $ 1< p < \infty $.  
Let $  A \in B_{p}$ and $  B\in B_{p'}$.  For all Calder\'on-Zygmund operators $ T$, there holds 
\begin{equation}\label{e:main}
\lVert   T _\sigma \::\: L ^{p} (\sigma ) \to L ^{p} (w)\rVert  \lesssim  C_{T} \Bigl\{\sup _{Q} 
\langle \sigma ^{\frac1{p'}}  \rangle_{\overline A, Q} \langle w ^{\frac1{p}}  \rangle _{\overline B,Q} 
\Bigr\}  .
\end{equation}
\end{priorResults}

In the inequality above, the bumps occur in the same product.  The logic of the two weight theory suggests that the bumps could appear individually. 
Cruz-Uribe and  Reznikov and  Volberg \cite{11120676} identified this conjecture as: 

\begin{conjecture}\label{c} Let $ \sigma $ and $ w$ be two weights with densities, and $ 1< p < \infty $.  
Let $  A \in B_{p}$ and $  B\in B_{p'}$.  For all Calder\'on-Zygmund operators $ T$, there holds 
\begin{equation}\label{e:c}
\lVert   T _\sigma \::\: L ^{p} (\sigma ) \to L ^{p} (w)\rVert \lesssim  C_{T}  \{[\sigma , w] _{\overline A,p'} + [ w, \sigma ] _{\overline B,p} \}   \lVert f\rVert_{L ^{p} (\sigma )}. 
\end{equation}
The constant $ C_T$ is defined in \eqref{e:CT}, and the implied constant depends upon the choice of $ A, B $.    
\end{conjecture}

This has been difficult to verify.  The very interesting paper of Nazarov, Reznikov and  Volberg \cite{13062653} proves a certain version of the 
conjecture in the case of $ p=2$, under an additional condition which they speculate is necessary. 
The result below holds in all $ L ^{p}$,  has fewer technical hypotheses, and contains \cite{13062653}.  

\begin{theorem}\label{t:main}  
Let $ \sigma $ and $ w$ be two weights with densities,    and $ 1< p < \infty $. Let $  A \in B_{p}$ and $  B\in B_{p'}$, 
and let $ \varepsilon _p , \varepsilon _ {p'}$ be two monotonic increasing functions on $ (1, \infty )$ which satisfy 
$ \int _{1} ^{\infty } \varepsilon _{p} (t) ^{-p'} \frac {dt} t = 1$, and similarly for $ \varepsilon _{p'}$.  
Define 
\begin{equation}\label{e:[[}
 \lceil\sigma , w\rceil_{\overline A, \varepsilon _ p, p'} := 
 \sup _{ \textup{$ Q$ a cube}} 
\varepsilon _{p}  
\Bigl(1+
\frac { \langle \sigma ^{1/{p'}}  \rangle_{\overline A, Q}} {\langle \sigma  \rangle_Q ^{1/{p'}}}
\Bigr) 
\langle \sigma ^{\frac1{p'}}  \rangle_{\overline A, Q} \langle w \rangle_ Q ^{\frac 1p}. 
\end{equation}
For any Calder\'on-Zygmund operator, there holds 
\begin{equation*}
\lVert   T _\sigma \::\: L ^{p} (\sigma ) \to L ^{p} (w)\rVert \lesssim  C_{T}  \bigl\{\lceil\sigma , w\rceil _{\overline A, \varepsilon _{p}, p' }  +  \lceil w, \sigma \rceil _{\overline B, \varepsilon _{p'}, p,} \bigr\}    \lVert f\rVert_{L ^{p} (\sigma )}. 
\end{equation*}
The constant $ C_T$ is defined in \eqref{e:CT}. 
\end{theorem}

Compared to the condition in \eqref{e:c}, the bound above penalizes the product $ \langle \sigma ^{\frac1{p'}}  \rangle_{\overline A, Q} \langle w \rangle_ Q ^{1/p}$ by the additional term involving $ \varepsilon _p$. 
Now, $ \varepsilon_p (t)$ has to satisfy an integrability condition that requires it to be, for example, of the order of 
\begin{equation*}
\varepsilon (t) \simeq (\log t) ^{+\frac 1 {p'}}  (\log \log t) ^{+\frac 1 {p'}}   (\log\log\log t) ^{+\frac {1+ \eta } {p'}}, \qquad t \to \infty , \ \eta >0. 
\end{equation*}

We prefer this formulation of the theorem, since it requires fewer hypotheses than \cite{13062653}. 
In addition, the  extra hypothesis  involving $ \varepsilon _p $ and $ \varepsilon _{p'}$ enter the proof  in transparent, and seemingly unavoidable, way, see Lemma~\ref{l:3}. 
Our approach can provide further  insights into how to build counterexamples to Conjecture~\ref{c}, 
in light of the fact that the Bellman approach followed in \cite{11120676}, is not  `reversible' in the separated bump setting,

\bigskip

The study of `bumped' $ A_p$ conditions have an  extensive history within just the subject of weighted estimates for singular integrals. 
Related conditions have arisen in the spectral theory of Schr\"odinger operators \cite{MR800004}, and in the study of Sarason's conjecture 
on the composition of Toeplitz operators \cites{sarasonConj,MR1395967}.  The article of Neugebauer \cite{MR687633} firmly introduced the 
subject into weighted estimates, and P{\'e}rez \cites{MR1260114,MR1291534} promoted their fine study.  The reader is referred  to \cites{11120676,MR1617653,MR2351373,MR2797562,MR2854179} for more history on this subject. 

More recent developments have been as follows.  The paper \cite{11120676} raised as a conjecture the Theorem proved above, and established 
variants of the result where the $ A (t)$ were `log bumps', namely $ A (t) \simeq t ^{p} ( \log (\operatorname e ^{} + t)) ^{-1- \delta }$, for $ \delta >0$.  
The papers of Lerner \cite{MR3085756}*{Thm 1.5} and Nazarov,  Reznikov \& Volberg \cite{MR3127385} ($ p=2$) established 
Theorem~\ref{t:bumps}. 
The latter  estimate is strictly stronger than \eqref{e:main}, as an example of \cite{13082026}*{\S7} shows.  
Variants of these results were proved in the setting of spaces of homogeneous type in \cite{13082026}.

The next section collects standard aspects of the subject, and can be referred back to as needed. The proof of the main theorem is in \S\ref{s:proof}.

\section{Notation, Background} 

Constants are suppressed:  By $ A \lesssim B$, it is meant that there is an absolute constant  $ c$ so that $ A \leq c B$.

We say that $ K \::\: \mathbb R ^{d} \times \mathbb R ^{d} \to \mathbb R $ is a \emph{Calder\'on-Zygmund kernel} if for some constants $ \Lambda >0$, 
and  $ C_K> 0$, and $ 0< \eta < 1$, such that  these conditions hold: For $ x, x', y \in \mathbb R ^{d}$
\begin{gather*}
\lVert K ( \cdot , \cdot )\rVert_{\infty } < \infty , \qquad  K (x,y)=0 , \quad \lvert  x-y\rvert> \Lambda , 
\\
\lvert  K (x,y)\rvert < C_K \lvert  x-y\rvert ^{-d} \,, \qquad x\neq y, 
\\
\lvert  K (x,y) - K (x',y)\rvert < C_K \frac {\lvert  x-x'\rvert ^{\eta } } { \lvert  x-y\rvert ^{d+ \eta } } , \qquad  \textup{if $ $}
2 \lvert  x-x'\rvert < \lvert  x-y\rvert  ,
\end{gather*}
and a fourth condition, with the roles of the first and second coordinates of $ K (x,y)$ reversed also holds.  These are typical conditions, 
although in the first condition, we have effectively truncated the kernel, at the diagonal and infinity. The effect of this is that we needn't be concerned 
with principal values.  

Given a Calder\'on-Zygmund kernel $ K$ as above, we can define 
\begin{equation*}
T f (x) := \int K (x,y) f (y) \; dy 
\end{equation*}
which is defined for all $ f\in L ^2 $ and $ x\in \mathbb R ^{d}$. We say that $ T$ is a \emph{Calder\'on-Zygmund operator}, since it necessarily extends to a bounded operator on $ L ^2 (\mathbb R ^{d})$.  We define 
\begin{equation} \label{e:CT}
C_T := C_K + \lVert T \::\: L ^2 \to L ^2 \rVert_{}. 
\end{equation}
It is well-known that $ T$ is also bounded on $ L ^{p}$, $ 1< p < \infty $, with norm controlled by $ C_T$.

\bigskip

To say that $ A \::\: [0, \infty ) \to [0, \infty )$ is a \emph{Young function} means that $A (0)=0$, $ A (1)=1$, $ A$ is convex, continuous increasing  and $ A (t)/t \to \infty $, as $ t \to \infty $.  The \emph{Luxembourg norm} associated to $ A$, on a measure space $ (X, \mu )$ is given by 
\begin{equation*}
\lVert f\rVert_{L ^{A} (X, \mu )} := \inf \{ \lambda >0 \::\:  \int _{X} A (\lvert  f\rvert/\lambda  )\, d \mu \le 1\}. 
\end{equation*}
The infimum of the empty set is taken to be $ \infty $.   

We most frequently work with Young functions of the form $ A (x)$ of the form $ c t ^{p} \alpha (t )$, where $ 1< p < \infty $, and 
$ \alpha (t) $ is constant for $ 0\le t \le 1$, and $ \alpha (t) $ is slowly varying at infinity, thus $ \alpha (2t)/\alpha (t) \to  1$ as $ t\to \infty $.  
The \emph{dual} to $ A$ is defined by $ \overline A (t)  \simeq  \int _{0} ^{t} ( A') ^{-1} (s) \; ds $. In this case, note that 
\begin{equation}\label{e:overline}
\overline  A (t) = \int _{0}  ^{t} \Bigl[\frac {s} {\alpha (s )} \Bigr] ^{\frac 1 {p-1}} \; ds 
\simeq t ^{p'} \alpha (t) ^{- \frac 1 {p-1}}.  
\end{equation}
This is also a Young function, and we will use the \emph{generalized H\"older's inequality}, in the form below, but note $ X$ must be a probability space.  
\begin{equation}\label{e:hi}
\int _{X} f \cdot g \, d \mu  \lesssim  \lVert f\rVert_{L ^{A} (X, \mu )} \lVert g\rVert_{L ^{\overline A } (X, \mu )} , \qquad \mu (X)=1.  
\end{equation}

\section{General Considerations} 

We use the following consequence of Lerner's remarkable median inequality.  We say that 
$ T $ is a \emph{sparse operator} if
$
T   f=\sum_{Q\in \mathcal Q} \langle f \rangle_Q \mathbf 1_{Q}, 
$
where $ \mathcal Q$ is	 a collection of dyadic cubes for which    
\begin{equation} \label{e:1/2}
\Bigl\lvert 
\bigcup _{Q' \in \mathcal Q \::\: Q'\subsetneq Q} Q' 
\Bigr\rvert  \le \tfrac {1} {2} . 
\end{equation}
We will also refer to $ \mathcal Q$ as \emph{sparse}, and will typically suppress the dependence of $ T$ on $ \mathcal Q$.  
Trivially, any subset of a sparse collection is sparse.  
(The definition of sparseness in \cite{13062653} is a bit more general, but this is not needed herein.) 

A sparse operator is bounded on all $ L ^{p}$, and in fact, is a `\emph{positive} dyadic Calder\'on-Zygmund operator.'  
And the class is sufficiently rich to capture the norm behavior of an arbitrary Calder\'on-Zygmund operator, 
which is the wonderful observation of  Andrei Lerner: 

\begin{priorResults}\label{t:}\cite{MR3127380}*{Thm 1.1}  We have the equivalence 
\begin{equation*}
\sup _{T \in CZO} \sup _{\lVert f\rVert_{L ^{p} (\sigma )} =1} \lVert T _{\sigma } f\rVert_{L ^{p} (w)} 
\simeq  \sup _{\textup{ $ T$ is sparse}}
\sup _{\lVert f\rVert_{L ^{p} (\sigma )} =1}  
\lVert  T_\sigma f  \rVert_{L ^{p} (w)}  . 
\end{equation*}

\end{priorResults}

We need this consequence of the two weight theory 
of E. Sawyer \cite{MR930072}. 
A detailed study of the dyadic setting is in \cite{09113437}, and an efficient proof is given at the end of \cite{12123840}.

\begin{priorResults}\label{l:test} \cite{MR930072} For a sparse operator $ T$ with sparse collection $ \mathcal Q$, the   norm is approximately 
\begin{equation*}
\lVert T _{\sigma } \::\: L ^{p} (\sigma ) \to L ^{p } (w )\rVert_{}   \simeq N _{\mathcal Q}
\end{equation*}
where $ N _{\mathcal Q}$ is the best constant $ N$ in the inequalities 
\begin{equation}\label{e:testing}
\begin{split}
\int _{Q_0}  \Bigl\lvert   \sum_{Q\in \mathcal Q  \::\: Q\subset Q_0}  \langle \sigma \rangle_Q  \mathbf 1_{Q} \Bigr\rvert ^p \; d w 
\lesssim  N _{\mathcal Q} ^{p} \sigma  (Q_0), \qquad \textup{$ Q_0 \in \mathcal Q$,} 
\\
\int _{Q_0}  \Bigl\lvert   \sum_{Q\in \mathcal Q  \::\: Q\subset Q_0}  \langle w \rangle_Q  \mathbf 1_{Q} \Bigr\rvert ^{p'} \; d \sigma  
\lesssim  N _{\mathcal Q} ^{p'} w  (Q_0), \qquad \textup{$ Q_0 \in \mathcal Q$.} 
\end{split}
\end{equation}
\end{priorResults} 

\section{Proof of Main Theorem} \label{s:proof}

We consider the first group of testing inequalities in \eqref{e:testing}, and the proof given will, by duality, apply to the second testing inequalities.  
Assume for the remainder of the argument that $ Q_0$ is a fixed cube, 
$ \mathcal Q$ is a sparse collection of cubes contained in $ Q_0$,  $ \mathcal Q$ contains $ Q_0$. 
In fact, we strengthen the sparseness condition \eqref{e:1/2} to 
\begin{equation} \label{e:1/10}
\Bigl\lvert \bigcup _{Q' \in \mathcal Q \::\: Q'\subsetneq Q} Q'  \Bigr\rvert < 10 ^{-p } \lvert  Q\rvert, \qquad Q\in \mathcal Q .
\end{equation}
This can be accomplished by dividing $ \mathcal Q$ meeting the condition \eqref{e:1/2} into a bounded number (as a function of $1< p < \infty $) 
of subcollections.  
Set 
\begin{equation} \label{e:rho}
\rho (Q) := \frac { \langle \sigma ^{1/{p'}}  \rangle_{\overline A, Q}} {\langle \sigma  \rangle_Q ^{1/{p'}}} \ge 1 . 
\end{equation}
These are the ratios that appear in the argument of  $ \varepsilon _p (t)$. 
In addition, we are free to assume that $  \varepsilon _{p} (t) $ does not grow faster than a power of $ \log t$ at infinity.  

Assume that for all  $ Q\in \mathcal Q$, 
\begin{equation}\label{e:alpha} 
\varepsilon _p (\rho (Q))\langle \sigma ^{1/p'} \rangle _{\overline A,  Q} \langle w \rangle_Q ^{1/p}
\simeq 
\varepsilon _p (\rho (Q_0))\langle \sigma ^{1/p'} \rangle _{\overline A,  Q_0} \langle w \rangle_{Q_0} ^{1/p} := \alpha . 
\end{equation}
We will show that 
\begin{equation}
\int _{Q_0} 
 \Bigl\lvert   \sum_{Q\in \mathcal Q   }  \langle \sigma \rangle_Q  \mathbf 1_{Q} \Bigr\rvert ^p \; d w  
 \lesssim \alpha ^{p} \sigma (Q_0).  
\end{equation}
This is enough to conclude the first of the two testing inequalities in \eqref{e:testing}, by a trivial summation of $ \alpha $ over 
a geometrically decreasing sequence of values. 

We invoke a version of the \emph{parallel corona}, originating in \cite{12014319}, also see the last page of \cite{12123840}.
The first step is proceed to \eqref{e:testing} by duality. Thus,  for $ g$ a 
 non-negative function in $ L ^{p'} (Q_0, w)$, we will bound the form 
\begin{equation}\label{e:alpha'}
 \sum_{Q\in \mathcal Q   }  \langle \sigma \rangle_Q  \langle w \rangle_{Q}
 \langle g \rangle ^{w} _{Q} \cdot \lvert  Q\rvert \lesssim \alpha \sigma (Q_0) ^{1/p} \lVert g\rVert_{L ^ {p'} (w)}. 
\end{equation}
Here $\langle g \rangle ^{w} _{Q} = w (Q) ^{-1} \int _{Q} g \; w (dx) $.  

Construct stopping data for $ g$ as follows.   For any cube $ Q \in \mathcal Q$, set $ \textup{ch} _{\mathcal T}(Q)$ to be the largest subcubes $Q'\in \mathcal Q$, contained in $ Q$ such that $ \langle g \rangle ^{w} _{Q'} > 10 \langle g \rangle ^{w} _{Q}  $.  
Then, set $ \mathcal T := \bigcup _{k=0} ^{\infty } \mathcal T _{k}$ where $\mathcal T_0 := \{Q_0\}$,and inductively 
$ \mathcal T _{k+1} := \bigcup _{Q\in \mathcal T _{k}} \textup{ch} _{\mathcal T} (Q)$.  
Define for $ T\in \mathcal T$,  
\begin{equation}\label{e:Tsharp}
T _{\sharp} := T \setminus \bigcup _{T'\in \textup{ch} _{\mathcal T} (T) } T' 
\end{equation}
It follows that $ w (T _{\sharp}) \ge \tfrac 12 w (T)$.  Define 
\begin{equation}\label{e:gT}
 g_T = \langle g \rangle ^{w} _{T} \mathbf 1_{T _{\sharp}} + 
 \sum_{T'\in \textup{ch} _{\mathcal T} (T) } \langle g \rangle ^{w} _{T'} \mathbf 1_{T'}.  
\end{equation}
Let $ Q ^{t} $ be the minimal element of $ \mathcal T$ that contains $ Q$. Observe that if $ Q ^{t}=T$, then  
$ \langle g \rangle ^{w}_Q = \langle g_T\rangle_{Q} ^{w}$. Moreover, by the maximal function estimate 
\begin{align} \label{e:gT<}
 \sum_{T\in \mathcal T}  (\langle g \rangle ^{w} _{T}) ^{p'} w(T) 
\lesssim \lVert g\rVert_{L ^ {p'} (w)}  ^{p'}.
\end{align}

\smallskip 

The construction of stopping data for $ \sigma $ is \emph{not} dual. For any $ Q\in \mathcal Q$, set  $ \textup{ch} _{\mathcal S}(Q)$ to the 
largest subcubes 
$Q'\in \mathcal Q$, contained in $ Q$ such that $ \langle \sigma \rangle  _{Q'} > 10 \langle \sigma \rangle  _{Q}  $ \emph{or $ Q'\in \mathcal T$.}
(We restart the stopping cubes at each member of $ \mathcal T$!) 
Then, set $ \mathcal S := \bigcup _{k=0} ^{\infty } \mathcal S _{k}$ where $\mathcal S_0 := \{Q_0\}$,and inductively 
$ \mathcal S _{k+1} := \bigcup _{Q\in \mathcal S _{k}} \textup{ch} _{\mathcal S} (Q)$.  
Let $ Q ^{s} $ be the minimal element of $ \mathcal S$ that contains $ Q$.

   This is then the basic estimate.  

\begin{lemma}\label{l:S} 
 There holds, uniformly in $ T\in \mathcal T$, 
\begin{equation} \label{e:S}
\sum_{Q \::\: Q ^{t} = T}
\langle \sigma  \rangle_Q \langle w \rangle_Q\cdot \lvert  Q\rvert 
\lesssim 
\alpha 
\Bigl[ 
\sum_{Q \::\: Q ^{t} = T } 
\langle \sigma ^{1/p}  \rangle _{A, Q}  ^{p} \lvert  Q\rvert 
\Bigr] ^{1/p} w (T) ^{1/p'} . 
\end{equation}
\end{lemma}

\begin{proof}[Proof of \eqref{e:alpha'}] 
Denoting  $ Q ^{\ast} := (Q ^{s} , Q ^{t})$,  observe that the left side of \eqref{e:alpha'} is 
\begin{align}  \label{e:A}
\sum_{S \in \mathcal S} \sum_{T\in \mathcal T} 
\sum_{Q \::\: Q ^{\ast} = (S,T)} 
 \langle \sigma \rangle_Q   \langle w \rangle_Q \langle g \rangle ^{w} _{Q} \cdot \lvert  Q\rvert. 
\end{align}
The average $  \langle g \rangle ^{w} _{Q}  = \langle g_T \rangle ^{w} _{Q}$.  
We must have $ S\subset T$ since $ \mathcal T \subset \mathcal S$, but also 
we can restrict to $ S ^{t}=T$, for otherwise $ Q ^{t} \subsetneq T$.  
Therefore, the sum above is dominated by 
\begin{align*}
\eqref{e:A} & 
= \sum_{T\in \mathcal T}  \langle g_T \rangle ^{w} _{Q}  
\sum_{Q \::\: Q ^{t} = T}
\langle \sigma  \rangle_Q \langle w \rangle_Q\cdot \lvert  Q\rvert 
\\& 
\lesssim 
\alpha  \sum_{T\in \mathcal T}    \langle g_T \rangle ^{w} _{Q} 
\Bigl[ 
\sum_{Q \::\: Q ^{t} = T} 
\langle\sigma^{1/p}  \rangle _{A, Q} ^{p}  \lvert  Q\rvert 
\Bigr] ^{1/p} w (T) ^{1/p'} 
\end{align*}
Here, of course, we have applied \eqref{e:S}.  
Apply H\"older's inequality.  
We have   by sparseness and  P{\'e}rez' maximal function estimate \eqref{e:cmf}, 
\begin{align*}
\sum_{Q\in \mathcal Q}  
 \langle\sigma^{1/p}  \rangle _{A, Q}  ^p \lvert  Q \rvert  
 & \lesssim  \int _{Q_0}\sup _{Q} \langle\sigma^{1/p}  \rangle _{A, Q}  ^p \mathbf 1_{Q}  \; dx \lesssim \sigma (Q_0).  
\end{align*}
(Recall that $ \mathcal T \subset \mathcal S$.)
And, on the other, the estimate \eqref{e:gT<}. 

\end{proof}

The obstacle to an easy proof of \eqref{e:S} is that the cubes $ Q\in \mathcal Q$ need not be $ w$-Carleson.  
We give the proof in three steps, in which we sum only over certain cubes; the first two are very easy.  
Write $ \mathcal Q_0 = \{Q_0\}$, and inductively set $ \mathcal Q_{k+1}$ to be the maximal elements  $Q\in \mathcal Q  $
which are strictly contained in some $ Q'\in \mathcal Q _{k}$.  The basic fact, a consequence of \eqref{e:1/10}, is that for all $ k$, $ Q\in \mathcal Q_k$, 
\begin{equation}\label{e:Qsmall}
\sum_{\substack{Q'\in \mathcal Q _{k+j}\\ Q'\subset Q  }}  \lvert  Q'\rvert \lesssim 10 ^{-p j} \lvert  Q\rvert, \qquad j\in \mathbb N   
\end{equation}

We write $ i _{T} (Q)=k$ if for some integer $ \ell $, $ T\in \mathcal Q _{\ell }$ and $Q\subset T$ with $ Q\in \mathcal Q _{\ell +k}$. 
Our first easy estimate is as follows: We require that the average value of $ w$ is not too large in the sum below.  

\begin{lemma}\label{l:1} 
There holds, uniformly in $ T\in \mathcal T$, 
\begin{equation} \label{e:S1}
 \sum_{\substack{Q \::\: Q ^{t} = T \\ \langle w \rangle_Q  ^{1/p}< 2 ^{2 i _{T} (Q)} \langle w \rangle_T ^{1/p} }}
\langle \sigma  \rangle_Q \langle w \rangle_Q\cdot \lvert  Q\rvert 
\lesssim 
\alpha 
\Bigl[ 
\sum_{Q \::\: Q ^{t} = T } 
\langle\sigma^{1/p}  \rangle _{A, Q} ^{p} \lvert  Q\rvert 
\Bigr] ^{1/p} w (T) ^{1/p'} . 
\end{equation}
\end{lemma}

\begin{proof}
The restriction on $ w$ in the sum in \eqref{e:S1}, with the estimate \eqref{e:Qsmall}, immediately imply that 
\begin{equation*}
 \sum_{\substack{Q \::\: Q ^{t} = T \\ \langle w \rangle_Q  ^{1/p}< 2 ^{2 i _{T} (Q)} \langle w \rangle_T ^{1/p} }}
w (Q) \lesssim w (T).  
\end{equation*}
Use the inequality below, a consequence of the generalized H\"older's inequality \eqref{e:hi}, 
\begin{align}
\langle \sigma  \rangle_Q \langle w \rangle_Q 
&\lesssim  \langle\sigma^{1/p}  \rangle _{A, Q} 
\{  \langle\sigma^{1/p'}  \rangle _{ \overline {A}, Q} \langle w \rangle_Q ^{1/p} \} \langle w \rangle_Q ^{1/p'} 
\\&  \label{e:zs2w}
\lesssim \alpha  \langle\sigma^{1/p}  \rangle _{A, Q}  \langle w \rangle_Q ^{1/p'} . 
\end{align}
(We ignore the term $ \varepsilon _p (\rho (Q))$ here.)  
Thus, we need only bound the restricted sum over $ Q$ of the terms 
\begin{equation*}
 \langle\sigma^{1/p}  \rangle _{A, Q}   \lvert  Q\rvert ^{1/p} w (Q) ^{1/p'} . 
\end{equation*}
It is clear that H\"older's inequality  completes the proof.  
\end{proof}

In our second easy estimate, we impose the restriction that the average of $ \sigma $ has become rather small.  

\begin{lemma}\label{l:2} 
There holds, uniformly in $ T\in \mathcal T$, 
\begin{equation} \label{e:S2}
\sum_{S\in \mathcal S \::\: S ^{t}=T } \sum_{\substack{Q \::\: Q ^t =S\\ \langle \sigma  \rangle_Q  ^{1/p}< 2 ^{-i _{S} (Q)/2} \langle \sigma \rangle_S ^{1/p}}}
\langle \sigma  \rangle_Q \langle w \rangle_Q\cdot \lvert  Q\rvert 
\lesssim 
\alpha 
\Bigl[ 
\sum_{S\in \mathcal S \::\: S ^{t}=T }  
\langle\sigma^{1/p}  \rangle _{A, Q} ^{p} \lvert  Q\rvert 
\Bigr] ^{1/p} w (T) ^{1/p'} . 
\end{equation}
\end{lemma}

\begin{proof}
The restriction on $ \sigma $ in the sum in \eqref{e:S2}  immediately imply that for each $ S$
\begin{equation*}
\sum_{\substack{Q \::\: Q ^{s} = S\\ \langle \sigma  \rangle_Q  ^{1/p}< 2 ^{-i _{S} (Q)/2} \langle \sigma \rangle_S ^{1/p}}}
\langle \sigma  \rangle_Q   \mathbf 1_{Q}\lesssim \langle \sigma  \rangle_S \mathbf 1_{S}.   
\end{equation*}
Then, estimate 
\begin{align}
\textup{LHS} \eqref{e:S2} 
& =  \int _{T} 
\sum_{S\in \mathcal S \::\: S ^{t}=T } \sum_{\substack{Q \::\: Q ^{s} = S\\ \langle \sigma  \rangle_Q  ^{1/p}< 2 ^{-i _{S} (Q)/2} \langle \sigma \rangle_S ^{1/p}}}
\langle \sigma  \rangle_Q   \mathbf 1_{Q} \; d w 
\\
&\lesssim 
\int _{T}
\sum_{S\in \mathcal S \::\: S ^{t} = T}
\langle \sigma  \rangle_S    \mathbf 1_{S}\; d w 
\\ \label{e:pass}
& \lesssim w (T) ^{1/p'} 
\Bigl[ 
\int _{T} \Bigl\lvert 
\sum_{S\in \mathcal S \::\: S ^{t} = T}
\langle \sigma  \rangle_S    \mathbf 1_{S}
\Bigr\rvert ^p \; d w 
\Bigr] ^{1/p} 
\\
& \lesssim w (T) ^{1/p'}  
\Bigl[ 
\int _{T}  
\sum_{S\in \mathcal S \::\: S ^{t} = T}
\langle \sigma  \rangle_S  ^p   \mathbf 1_{S}
\; d w 
\Bigr] ^{1/p} 
\end{align}
This just depends upon the fact that for each $ x$, the sequence of numbers $ \langle \sigma  \rangle_S	 \mathbf 1_{S} (x)$ 
is growing at least geometrically, so that we can pass the $ p$th power inside the sum. 

Continuing, we note that as a consequence of \eqref{e:zs2w}, 
\begin{align*}
\langle \sigma  \rangle_S  ^p  w (S) &= \langle \sigma  \rangle_S  ^p  \langle w \rangle_S \cdot \lvert  S\rvert 
\lesssim \alpha ^p  \langle\sigma^{1/p}  \rangle _{A, S} ^p \lvert  S\rvert.   
\end{align*}
And this completes the proof.  (Again, the role of $ \varepsilon _p$ is ignored.)
\end{proof}

\begin{remark}\label{r:pass} The argument above, pushing the $ p$th power inside the sum in \eqref{e:pass} is 
key. It for instance encapsulates the proof of Theorem~\ref{t:cmf}: Combine \eqref{e:cmf} with Sawyer's two weight maximal function 
theorem \cite{MR676801}, using the argument \eqref{e:pass} to verify the testing condition.  
But, there is nothing quite so simple in the current setting---the sparse operators are just a bit bigger than the maximal function. 
Nevertheless, we want to appeal to an argument similar to \eqref{e:pass} below. To do so, we will absolutely need a certain 
integrability condition on the function $ \varepsilon _{p} (t)$. 
\end{remark}

The third, and decisive, sum is over cubes complementary to the previous two.  

\begin{lemma}\label{l:3} 
There holds, uniformly in $ T\in \mathcal T$, 
\begin{equation} \label{e:S3}
\sum_{S\in \mathcal S \::\: S ^{t}=T } \ 
\sideset {} {^3} \sum_{\substack{Q \::\: Q ^s = S}} 
\langle \sigma  \rangle_Q \langle w \rangle_Q\cdot \lvert  Q\rvert 
\lesssim 
\alpha 
\Bigl[ 
\sum_{Q \::\: Q ^{t} = T } 
\langle\sigma^{1/p}  \rangle _{A, Q}  \lvert  Q\rvert 
\Bigr] ^{1/p} w (T) ^{1/p} . 
\end{equation}
Above the superscript $  ^{3}$ on the second sum means that we sum over those $ Q $ such that 
\begin{equation}\label{e:3}
\langle w \rangle_Q ^{1/p} \ge 2 ^{2 i _{T} (Q)} \langle w \rangle_T  ^{1/p}, \quad \textup{and} \quad 
\langle \sigma  \rangle_Q ^{1/p} \ge  2 ^{-i _{S} (Q)/2} \langle \sigma \rangle_S ^{1/p} 
\end{equation}
\end{lemma}

\begin{proof}
The role of $ \varepsilon _p$ is now key.  
The main point of the assumptions on the sum in this case is that the ratios $ \rho (Q)$ \emph{decrease at least geometrically, holding the stopping parents fixed:} 
\begin{equation}\label{e:decrease}
\rho (Q) \lesssim 2 ^{- i _{T} (Q)/2 } \rho (T), \qquad Q\in \mathcal Q ,\  Q ^s=S,\ Q ^{t}=T.   
\end{equation}
Indeed, exploiting the definitions \eqref{e:rho}, \eqref{e:alpha},  and the conditions in \eqref{e:3},  
\begin{align*}
\varepsilon_p (\rho (Q)) ^{-1} \alpha  & \simeq   \langle \sigma ^{1/p'} \rangle _{\overline A, Q}  \langle w \rangle_Q ^{1/p}  
 & \eqref{e:alpha} 
\\&
\simeq
\rho (Q) \langle \sigma  \rangle _{Q} ^{1/p'} \langle w \rangle_Q ^{1/p}  & \eqref{e:rho}
\\
& \gtrsim 2 ^{- i _{S} (Q)/2+ 2 i_T (Q)}  \rho (Q) \langle \sigma  \rangle_S ^{1/p'}\langle w \rangle_T ^{1/p}  & \eqref{e:3} 
\\
& \gtrsim   2 ^{ i_T (Q)} \rho (Q)  \langle \sigma  \rangle_S ^{1/p'} \langle w \rangle_T  ^{1/p} 
& ( i _S (Q) \le i_T (Q)) 
\\
& \gtrsim   2 ^{ i_T (Q)} \rho (Q)  \langle \sigma  \rangle_T ^{1/p'} \langle w \rangle_T  ^{1/p} 
 & \textup{(construction of $ \mathcal S$)} 
\\
& \simeq 2 ^{i_T (Q)} \frac {\rho (Q) } { \rho (T)}  
\langle \sigma  \rangle_{\overline  A , T} ^{1/p'} \langle w \rangle_T  ^{1/p} & \eqref{e:rho}
\\
&\simeq \alpha  2 ^{i_T (Q)} \frac {\rho (Q) } { \rho (T)}   \varepsilon_p (\rho (T)) ^{-1}.  & \eqref{e:alpha} 
\end{align*}
Summarizing, there holds 
\begin{equation*}
2 ^{- i_T (Q)} \gtrsim \frac {\rho (Q) } { \rho (T)}    \frac {\varepsilon_p (\rho (Q)) } { \varepsilon_p (\rho (T)) } 
\gtrsim \Bigl[\frac {\rho (Q) } { \rho (T)}   \Bigr] ^2 ,  
\end{equation*}
since we assume that $ \varepsilon _p (t)$ grows more slowly than a power of log at infinity.  This proves \eqref{e:decrease}. 

The integrability condition on the function $ \varepsilon_p $ is used this way. 
By the monotonicity of $ \varepsilon _p (t)$,  the  condition $ \int ^{\infty }_1 \varepsilon (t) ^{-p'} \frac {dt} t =1 $, and the geometric decay \eqref{e:decrease}, we see that 
\begin{equation} \label{e:sum}
\sum_{S\in \mathcal S \::\: S ^{t}=T } \ \sideset {} {^{3}}\sum_{Q \::\: Q ^s=S} \varepsilon (\rho (Q)) ^{-p'} \mathbf 1_{Q} (x) \lesssim \mathbf 1_{T} (x). 
\end{equation}

We will rely upon this estimate below, which holds for all $ Q\in \mathcal Q$. 
\begin{align}
\langle \sigma  \rangle_Q ^{p} \langle w \rangle_Q 
&\lesssim \langle\sigma ^{1/p}   \rangle _{A, Q} ^{p} \langle  \sigma ^{1/p'}   \rangle _{\overline { A}, Q}  ^{p}  \langle w \rangle_Q
\\& \simeq  \varepsilon_p (\rho (Q)) ^{-p}  \langle\sigma^{1/p}   \rangle _{A, Q} ^{p}  \{ \varepsilon_p (\rho (Q)) ^{p} \langle\sigma^{1/p'}   \rangle _{\overline  A, Q} ^{p}   \langle w \rangle_Q \} 
\\  \label{e:sw<}
& \simeq  
\alpha ^{p}  \varepsilon (\rho (Q)) ^{-p} \langle\sigma^{1/p}   \rangle _{A, Q}  ^{p} . 
\end{align}

Now, the proof can be completed.  We can estimate as below, inserting $ \varepsilon (\rho (Q))$ to the zero power, and applying H\"older's inequality. 
\begin{align*} 
\textup{LHS} \eqref{e:S3} 
& = \int _{T} 
\sum_{S\in \mathcal S \::\: S ^{t}=T } \ 
\sideset {} {^3} \sum_{\substack{Q \::\: Q ^{\ast} = (S,T)}}  \varepsilon (\rho (Q)) ^ {-1+1}
\langle \sigma  \rangle_Q \mathbf 1_{Q} \; d w 
\\
& \lesssim  
w (T) ^{1/p'} 
\Bigl[
\int _{T}  
\sum_{Q \::\: Q ^{t}=T}  \varepsilon (\rho (Q)) ^{p}
\langle \sigma  \rangle_Q  ^{p} \mathbf 1_{Q} 
 \; dw 
\Bigr] ^{1/p}   & \textup{by  \eqref{e:sum}}
\\ 
& = 
w (T) ^{1/p'} 
\Bigl[ 
\sum_{Q \::\: Q ^{t}=T}  \varepsilon (\rho (Q)) ^{p} \langle \sigma  \rangle_Q  ^{p} \langle w \rangle_{Q} \lvert  Q\rvert   
\Bigr] ^{1/p} 
\\
& \lesssim  \alpha 
w (T) ^{1/p'} 
\Bigl[ 
\sum_{Q \::\: Q ^{t}=T}   \langle \sigma ^{1/p} \rangle_{A, Q}  ^{p} \lvert  Q\rvert   
\Bigr] ^{1/p}     & \textup{by \eqref{e:sw<}.}
\end{align*}
This completes the proof. 
\end{proof}

\section{Log and LogLog Bumps} 

We illustrate the implications of our main theorem for certain `borderline' bumps.  
By a \emph{log bump}, we mean function in  $ B_p$  of the form 
\begin{equation*}
L _{p', \eta }(t) \simeq t ^{p} (\operatorname{Log}  t ^{p'}) ^{-1-  (p-1)\eta }, \qquad \eta >0, \ t \to \infty . 
\end{equation*}
In this case, it follows from \eqref{e:overline} that  the dual Young function is $ \overline {L _{p',\eta }} (t) \simeq t ^{p'} (\operatorname{Log}  t ) ^{\frac 1 {p-1}+ \eta  }$.  
Throughout this section, $\operatorname{Log}  t  := 1+ \max \{0, \log t\} $. 

By a \emph{log-log bump}, we mean functions  in $  B_p$  of the form 
\begin{gather*}
\Lambda _{p',\eta } (t) \simeq t ^{p}  (\operatorname{Log}  t) ^{-1} (\operatorname{Log}  t ) ^{-1-  (p-1)\eta }, \qquad \eta >0, \ t \to \infty . 
\\
\overline {\Lambda _{p',\eta }} (t) \simeq t ^{p'} (\operatorname{Log}  t)  ^{\frac 1 {p-1}} (\operatorname{Log}  \operatorname{Log}  t) ^{\frac 1{p-1}+ \eta  }
\end{gather*}

We show that all the log bumps fall under the scope of our main theorem, and that log-log bumps do as well, provided the power on the bump 
is sufficiently large.   

\begin{proposition}\label{p:eta}  These two assertions hold:\\ 
\begin{enumerate}
\item Suppose that $ \eta >0$, and that $ \sigma $ and $ w$ are a pair of weights such that 
 $ [\sigma , w] _{\overline {L _{p',\eta }} ,p' }< \infty $, as  defined in \eqref{e:[]}.  Then, for $ \varepsilon (t) = t ^{\eta /2 p'}$, 
we have $ \lceil\sigma , w\rceil_{\overline {L _{p',\eta /2}}, \varepsilon _ p, p' } < \infty $, where the latter expression is defined in \eqref{e:[[}.

\item Suppose that $ \eta > 0 $, and that $ \sigma $ and $ w$ are a pair of weights such that 
 $ [\sigma , w] _{\overline {\Lambda  _{p', 1+\eta }},p' }< \infty $.   Then,  
set  $ \varepsilon_p (t)=  (\operatorname{Log} t) ^{ 1/p' + \eta '/p'} $, 
we have $ \lceil\sigma , w\rceil_{\overline {\Lambda_{p',\eta'/p}}, \varepsilon _ p, p'} < \infty $. 

\end{enumerate}
Moreover, in both cases, 
\begin{equation} \label{e:finite}
\int ^{\infty } \varepsilon_p (t) ^{-p'} \frac {dt} t < \infty. 
\end{equation}
\end{proposition}

This relies upon the observations in \cites{13062653,MR3127380} concerning self-improvement of Orlicz norms, and we will specialize their more general considerations.  
For the bump functions that we are interested in, write 
$
\overline  A (t) = t ^{p'} \beta (t ^{p'}), 
$
so that  for $ u$ sufficiently large, 
\begin{equation}
\beta (u)  \simeq 
\begin{cases}
(\operatorname{Log}  u ) ^{\frac 1{p-1}+ \eta }  & A = L _{p', \eta }
\\
(\operatorname{Log}  u) ^{\frac 1{p-1}} (\operatorname{Log}  \operatorname{Log}  u ) ^{\frac 1{p-1}+ \eta }  & A = \Lambda _{p',\eta }
\end{cases}
\end{equation}
Also, for a cube $ Q$, set 
\begin{equation*}
D _{Q} (\lambda ) := \frac 1 {\lvert  Q\rvert } \lvert  \{ x\in Q \::\: \sigma (x) > \lambda \}\rvert. 
\end{equation*}

Set $ B (t) = t \beta (t)$.  We have 
\begin{lemma}\label{l:B}  With the notations just established, there holds 
\begin{equation*}
\langle \sigma \rangle_{B, Q} \simeq \int _{0} ^{\infty } D _Q (\lambda ) \beta (D _Q (\lambda ) ^{-1}  ) \, d \lambda . 
\end{equation*}
\end{lemma}

\begin{proof}
Both sides scale as a norm. So, assuming that $ \lVert \sigma \rVert_{B, Q} =1$, 
we have 
\begin{align*}
1 = \frac 1 {\lvert  Q\rvert }\int _{Q}  A (\sigma (x))\; dx & = 
\int _{0} ^{\infty } A (t ) d D_Q (t ) 
\\
& \simeq \int _{0} ^{\infty } \beta (t )  \cdot D_Q (t) \; dt 
\end{align*}
by integration by parts.  Now, by Chebyshceff, $ D_Q (t) \le t ^{-1} $, hence 
\begin{equation*}
1 \lesssim \int _{0} ^{\infty } \beta ( D_Q (t)  ^{-1} )  \cdot D_Q (t) \; dt . 
\end{equation*}

For the reverse inequality, we need to see that $ \int ^{\infty } \beta ( D_Q (t)  ^{-1} )  \cdot D_Q (t) \; dt \lesssim 1 $. 
And, so we need only be concerned with those $ t$ for which 
$ \beta ( D_Q (t)  ^{-1} ) > C \beta (t)$.  But, for appropriate $ C$, this would imply that $ D_Q (t)  \le \min\{ 1,  t ^{-2}\} $, so the proof is finished.  
\end{proof}

\begin{definition}\label{d:wc}\cite{13062653}*{Def 2}
A function f is \emph{weakly concave} on an interval $ I$, if there is a constant $ C\ge 1 $ so that  for any numbers $ x_1 ,\dotsc, x_n\in I$ in its domain, and $ \lambda _1 ,\dotsc, \lambda _n\ge 0$,  with $ \sum_{j=1} ^{n} \lambda _j =1$, there holds 
\begin{equation*}
f \Bigl(  \sum_{j=1} ^{n} \lambda _j x_n \Bigr) \ge C  \sum_{j=1} ^{n} \lambda _j f(x_n). 
\end{equation*}
\end{definition}

\begin{priorResults}\label{t:wc}\cite{13062653}*{Thm 5.3}  Write $ B_0 (t) = t \beta _0 (t)$, and assume 
$ B_0 (t) \lesssim  B (t) \theta ( \beta  (t))$, where $ \theta (t) $ is decreasing with 
$t \to  t \theta (t)$  weakly concave on $ (0, \infty )$.  Then, 
\begin{equation*}
\langle \sigma \rangle_{B_0, Q} \le C \langle \sigma \rangle_{B, Q} \theta \Bigl( \frac { \langle \sigma \rangle_{B, Q}} { \langle \sigma \rangle_Q} \Bigr). 
\end{equation*}
\end{priorResults}

The proof of the theorem relies upon Lemma~\ref{l:B}, and a Jensen inequality for weakly concave functions. 
(The constant in the Jensen inequality depends upon the  constant in weak convexity.)

\begin{proof}
We can assume that $ \langle \sigma  \rangle_{Q} = 1$, so that $ \int _{0} ^{\infty } D_Q (\lambda )\; d \lambda =1$. 
There holds, by Lemma~\ref{l:B}, 
\begin{align*}
\langle \sigma \rangle_{B_0, Q}  &\simeq  \int _{0} ^{\infty  } D_Q (\lambda ) \beta _0 ( D_Q (\lambda ))\; d \lambda 
\\
& \lesssim  
 \int _{0} ^{\infty  } D_Q (\lambda ) \beta  (D_Q (\lambda )) \theta  ( \beta (D_Q (\lambda )))\; d \lambda 
\\
& = \int _{0} ^{\infty }  D_Q (\lambda )  \tilde \theta (  \beta  (D_Q (\lambda ) ) \; d \lambda   &  (\tilde \theta (t) := t \theta (t))
\\
& \lesssim  \tilde \theta \Bigl(  \int _{0} ^{\infty }  D_Q (\lambda )  \beta ( D_Q (\lambda)  )  \; d \lambda  \Bigr)
\\
& = \int _{0} ^{\infty }  D_Q (\lambda )  \beta ( D_Q (\lambda)  )  \; d \lambda  \times   \theta 
\Bigl( \int _{0} ^{\infty }  D_Q (\lambda )  \beta ( D_Q (\lambda)  )  \; d \lambda   \Bigr) 
\\&\simeq \langle \sigma  \rangle_{B,Q}  \theta ( \langle \sigma  \rangle _{B,Q}) . 
 \end{align*}
The first inequality is just hypothesis, and the second inequality is  the  Jensen inequaltiy which is an immediate corollary to weak concavity. 
Finally, the expression is rewritten.  This proves the Theorem. 
\end{proof}

\begin{proof}[Proof of Proposition~\ref{p:eta}] 
In the first case, we have $  [\sigma , w] _{\overline {L _{p',\eta }} ,p' }< \infty$.  Set 
\begin{align*}
B (u) := u (\operatorname{Log} u) ^{p'-1+ \eta }\,, \qquad  \theta (u) := (1+u) ^{- \eta /2}. 
\end{align*}
Then, set $ B_0 (u) := B (u) \theta ( \operatorname {Log} u)$, and $ \overline  {A_0} := B_0 (t ^{p})$.  From the definitions, 
$ A_0 (t) = L _{p',\eta /2}$, and 
\begin{align} \label{e:theta}
\begin{split}
\langle \sigma ^{1/p'}  \rangle_{\overline {A_0}, Q} ^{p'} &:= 
\inf \{ \lambda ^{p'} >0 \::\: \frac 1 {\lvert  Q\rvert } \int _{Q}  B_0 (\sigma / \lambda ^{p'}) \le 1\} 
= \langle \sigma  \rangle _{B_0, Q} 
\\
& \lesssim \langle \sigma \rangle_{B, Q}  \theta \Bigl( \frac { \langle \sigma \rangle_{B, Q}} {\sigma (Q)} \Bigr).
\end{split}
\end{align}
The last line follows from the Theorem~\ref{t:wc}, and the observation that $ t \to t \theta (t)$ is 
weakly concave.  

With $ \varepsilon _p (t) ^{p'} := \theta (t)  ^{-1}$, it follows from our assumptions 
that $ A_0 \in B_ p$, and moreover that 
\begin{equation}\label{e:ep}
\begin{split}
\varepsilon _p (t) ^{p'} 
\langle \sigma ^{1/p'} \rangle _{\overline  {A_0}, Q} ^{p'} \langle  w \rangle ^{p'/p} 
& \lesssim  \langle \sigma \rangle_{B, Q}  \langle  w \rangle ^{p'/p}  
\\
& = \langle \sigma ^{1/p'} \rangle _{\overline A, Q} ^{p'}  \langle  w \rangle ^{p'/p} \le  [\sigma , w] _{\overline {L _{p',\eta }} ,p' } ^{p'}.  
\end{split}
\end{equation}
And this completes this case, since the integral condition \eqref{e:finite} becomes, after a change of variables $ e ^{u} = t$, 
\begin{equation*}
\int ^{\infty } \varepsilon_p (t) ^{-p'} \frac {dt} t 
= \int ^{\infty }  \operatorname e ^{ - \eta  u /2} \; d u < \infty . 
\end{equation*}

\medskip 
In this case,  the assumption is  $  [\sigma , w] _{\overline {\Lambda  _{p', 1+\eta }} ,p' }< \infty$.  Recall that in this case we have 
\begin{equation*}
\overline  {\Lambda _{p', 1+ \eta }} (t) 
\simeq t ^{p'} (\operatorname {Log}t ^{p'} ) ^{\frac 1 {p-1}} (\operatorname {Log}\operatorname {Log} t ^{p'}) ^{ \frac 1 {p-1} + 1+ \eta }.  
\end{equation*}
Set 
\begin{align*}
B (u) := u (\operatorname{Log} u)^{\frac 1 {p-1}}  (\operatorname{Log}\operatorname{Log} u) ^{ {\frac 1 {p-1}} + \eta /2 }\,, \qquad  \theta (u) := (\operatorname{Log}  u) ^{-1- \eta' /2}. 
\end{align*}
Then, $ B_0 (u) := B (u) \theta (u)$, and $ \overline  {A_0} (t) := B (t ^{p})$.  Then, $ A_0 \in B_p$. 
The function $ t \mapsto t \theta (t)$ is 
weakly concave. (One can consult \cite{13062653}*{\S5.2.2}, or alternatively, check that  
  $ \theta $ is concave on $ [C_0, \infty )$, for sufficiently large $ C_0$.)   Thus, \eqref{e:theta} continues to hold in this case.  

With $ \varepsilon _p (t) ^{p'} := \theta (t)  ^{-1}$, the inequality \eqref{e:ep} will continue to hold, with obvious changes.  
And the integral condition \eqref{e:finite} holds, since with the same exponential change of variables, 
\begin{equation*}
\int ^{\infty } \varepsilon_p (t) ^{-p'} \frac {dt} t 
= \int ^{\infty }  \; \frac{ d u} {u ^{1+ \eta /2}} < \infty . 
\end{equation*}

\end{proof}


\end{document}